\documentclass[11pt, a4paper]{amsart}
\usepackage{amsfonts,amssymb,amsmath, amsthm}
\usepackage{mathrsfs, mathtools}
\usepackage{lmodern}
\usepackage{qsymbols}
\usepackage{latexsym}
\usepackage[noadjust]{cite}
\usepackage{pdfsync}
\usepackage{bm}
\usepackage{enumitem}

\newtheorem{thm}{Theorem}[section]
\newtheorem{lem}[thm]{Lemma}

\theoremstyle{definition}
\newtheorem{defn}[thm]{Definition}
\newtheorem{rem}[thm]{Remark}

\numberwithin{equation}{section}
\usepackage[plainpages=false,pdfpagelabels,backref=page,citecolor=red]{hyperref}
\usepackage{xcolor}
\hypersetup{
colorlinks,
linkcolor={cyan!90!black},
citecolor={magenta},
urlcolor={green!40!black}
} 
\newcommand{\R}{\mathbb{R}}
\newcommand{\IC}{\mathbb{C}}

\newcommand{\cH}{\mathcal{H}}

\newcommand{\dist}{\operatorname{dist}}

\renewcommand{\L}{\operatorname{L}} 
\renewcommand{\H}{\operatorname{H}}
\newcommand{\Lloc}{\L_{\operatorname{loc}}} 
\newcommand{\C}{\operatorname{C}} 
\renewcommand{\H}{\operatorname{H}} 

\newcommand{\ree}{{\mathbb{R}^{n+1}}}
\newcommand{\gradx}{\nabla_x}
\renewcommand{\div}{\operatorname{div}}

\renewcommand{\d}{\, \mathrm{d}} 
\renewcommand\Re{\operatorname{Re}}
\renewcommand\Im{\operatorname{Im}}
\DeclareMathOperator{\dom}{\mathsf{D}} 

\def\Xint#1{\mathchoice
{\XXint\displaystyle\textstyle{#1}}%
{\XXint\textstyle\scriptstyle{#1}}%
{\XXint\scriptstyle\scriptscriptstyle{#1}}%
{\XXint\scriptscriptstyle%
\scriptscriptstyle{#1}}%
\!\int}
\def\XXint#1#2#3{{\setbox0=\hbox{$#1{#2#3}{%
\int}$ }
\vcenter{\hbox{$#2#3$ }}\kern-.6\wd0}}
\def\barint{\,\Xint -} 
\def\bariint{\barint_{} \kern-.4em \barint}
\def\bariiint{\bariint_{} \kern-.4em \barint}
\renewcommand{\iint}{\int_{}\kern-.34em \int} 
\renewcommand{\iiint}{\iint_{}\kern-.34em \int} 

\title[Gaussian bounds for degenerate parabolic equations]{On fundamental solutions and Gaussian bounds for degenerate parabolic equations with time-dependent coefficients}

\author{Alireza Ataei}
\email{alireza.ataei@math.uu.se}
\address{Department of Mathematics, Uppsala University, S-751 06 Uppsala,
Sweden}

\author{Kaj Nystr\"{o}m}
\email{kaj.nystrom@math.uu.se}
\address{Department of Mathematics, Uppsala University, S-751 06 Uppsala,
Sweden}

\thanks{}

\date{\today}

\begin{document}

\maketitle

\begin{abstract}
We consider second order degenerate parabolic equations with real, measurable, and time-dependent coefficients. We allow for degenerate ellipticity dictated by a spatial $A_2$-weight. We prove the existence of a fundamental solution and derive Gaussian bounds. Our construction
is based on the original work of Kato \cite{Kato}.
\end{abstract}
\section{Introduction}  We consider parabolic operators of the form
\begin{eqnarray}\label{eq1deg+}
\cH u:=\partial_tu+\mathcal{L}u:= \partial_t u  -w^{-1} \div_{x} (A(x,t) \nabla_{x}u),\  (x,t) \in \R^n \times \R \eqqcolon \ree,
\end{eqnarray}
where the weight $w = w(x)$ is time-independent and belongs to the spatial Muckenhoupt class $A_2(\mathbb R^{n},\d x)$, and the coefficient matrix $A = A(x,t)$ is measurable with real entries and possibly depends on all variables. Degeneracy of $A$ is also dictated by the  weight $w$ in the sense that $A$ satisfies
\begin{equation}
\label{ellip}
 c_1|\xi|^2w(x)\leq A(x,t) \xi \cdot {\xi}, \qquad
 |A(x,t)\xi\cdot\zeta|\leq c_2w(x)|\xi| |\zeta|,
\end{equation}
for some $c_1, c_2\in (0,\infty)$ and for all $\xi,\zeta \in \mathbb R^{n}$, $(x,t) \in \R^{n+1}$. We refer to $[w]_{A_2}$ as the constant of the weight and to $c_1, c_2$ as the ellipticity constants of $A$. We will frequently refer to  $n$, $c_1$, $c_2$, and $[w]_{A_2}$ as the structural constants.

Equations and operators as in  \eqref{eq1deg+}  appear naturally in the study of the fractional powers of parabolic equations and anomalous diffusions, see \cite{LN} and the references therein, and in the context of heat kernels of Schrödinger equations with singular potential, see \cite{IKO}. For contributions to the study of local properties of the solutions to  $\cH u = 0$ and the Gaussian estimates, we refer to \cite{CS,CUR2}. Furthermore, recently in \cite{AEN} we, together with M. Egert, established the Kato (square root) estimate for $\cH$ allowing also for complex coefficients. While this may  be considered as of independent interest, the result proved here and the results of \cite{AAEN} will be combined in a forthcoming work to give a generalization of the work in \cite{AEN} to weighted parabolic operators as in \eqref{eq1deg+} satisfying \eqref{ellip}.

Given  $0<T <\infty$, we in this paper consider the Cauchy problem
\begin{align}\label{cauchy}
\mathrm{(i)}&\ \cH u= \partial_t u  -w^{-1} \div_{x} (A(x,t) \nabla_{x}u)=0\,\mbox{ in }\mathbb R^n\times (0,T),\notag\\
\mathrm{(ii)}&\ \lim_{t \to 0}u(x,t) = f(x).
\end{align}
The equation in (i) is interpreted in the weak sense and according to the following definition.  We refer to the bulk of the paper for definitions and the functional setting.
\begin{defn} A weak solution to \eqref{cauchy} (i) on $\mathbb R^n\times (0,T)$ is
a (real-valued) function $u\in \L^2_{{\rm loc}}((0,T], \H^1_{w,{\rm loc}}(\mathbb R^n))$
such that
\begin{equation}
\label{eq:weak}
\int_0^T\int_{\mathbb R^n} u(x,t){\partial_t\phi(x,t)}
\,{\rm d} w\,{\rm d}t
=\int_0^T\int_{\mathbb R^n} A(x,t)\nabla_x u(x,t)\cdot {\nabla_x\phi(x,t)}
\,{\rm d}x\,{\rm d}t
\end{equation}
for all $\phi\in C_0^\infty(\mathbb R^n\times (0,T))$.
\end{defn} The purpose of this note is to establish the existence of a kernel/fundamental solution
associated to $\cH$, to derive appropriate Gaussian upper bounds for the kernel  in the nature of the original (unweighted) estimates of Aronson \cite{Ar67}, and to use the kernel  to represent weak solutions to \eqref{cauchy}. The quantitative estimates derive will only depend on  $n$, $c_1$, $c_2$, and $[w]_{A_2}$, i.e., on the structural constants.

 Recall that in the case of uniform elliptic coefficients, i.e.,  $w\equiv 1$, the problem in  \eqref{cauchy} was studied in depth in \cite{Ar68}. In {\cite{Ar68}} Aronson considered the energy space
$\L^\infty([0,T], \L^2({\mathbb{R}}^n))\cap \L^2((0,T], \H^1({\mathbb{R}}^n))$, he proved that all solutions $u$ in this space have a trace $f\in \L^2({\mathbb{R}}^n)$, and the solution is uniquely determined by this trace. He obtained existence,  given initial data in $\L^2$, and hence he defined an evolution operator $\Gamma$ such that $u(\cdot,t)=\Gamma(\cdot,t)f$ for $t> 0$. In \cite{Ar67}, pointwise Gaussian estimates of the evolution operator are proved. This result allows one to define weak solutions to \eqref{cauchy} by the integral representation
\begin{align}\label{aro1}
u(x,t)=\int_{{\mathbb{R}}^n}K_t(x,y)f(y)\,{\rm d}y=\int_{{\mathbb{R}}^n}K(x,t,y,0)f(y)\,{\rm d}y,
\end{align}
for $f$ in various spaces of initial conditions, where $K$ is the kernel/fundamental solution associated to $\cH$. Uniqueness is proved in the class of the solutions satisfying
\begin{align}\label{aro2}
\int_{0}^T\int_{{\mathbb{R}}^n} e^{-a |x|^2} |u(x,t)|^2
\,{\rm d}x\,{\rm d}t<\infty,
\end{align}
for some $a>0$,  and existence whenever
$f\in \L^2(e^{-\gamma |x|^2}{\rm d}x)$. In particular, this result covers the case $f\in \L^p({\rm d}x)$, $2\le p\le \infty$.

Given  $x \in \R^n, t>0$, we introduce $w_t(x) =: w(B_{\sqrt{t}}(x))$ where $B_{\sqrt{t}}(x)$ is the Euclidean ball of radius ${\sqrt{t}}$ and center $x$ in $\mathbb R^n$. This note is devoted to the proof of the following result.

\begin{thm}\label{main} Given $f\in \L^2_w(\mathbb R^n)$ and $T>0$, there exists a unique weak solution to the problem in  \eqref{cauchy}, such that
 \begin{align}\label{cauchya}u\in\L^\infty([0,T], \L^2_w({\mathbb{R}}^n))\cap \L^2((0,T], \H_w^1({\mathbb{R}}^n)),
  \end{align}
  and
  \begin{align}\label{cauchyb}u(\cdot,t)\to f(\cdot)\mbox{ in $ \L^2_w(\mathbb R^n)$ as $t\to 0^+$}.
  \end{align} The unique solution $u$ can be represented as
\begin{align}\label{cauchy+}
u(x,t)= \int_{\R^n} K_t(x,y) f(y)\, w(y) \d y,\mbox{ for all $(x,t)\in\mathbb R^n\times (0,T)$,}
\end{align}
where  $K_t(x,y)=K(x,t,y,0)$ is the fundamental solution of $\cH$, satisfying
\begin{align}\label{cauchy+all}
 \int_{\R^n} K_t(x,y) w(y) \d y= 1,\mbox{ for all $(x,t)\in\mathbb R^n\times (0,T)$.}
\end{align}
  Furthermore, there exist $c$,
$1\leq c<\infty$, and $\nu>0$, both depending only on the structural constants,  such that
\begin{equation}\label{est1}
K_t(x,y) \leq \frac{c}{\sqrt{w_t(x) w_t(y)}}
e^{-\frac{|x-y|^2}{ct}},
\end{equation}
for all $t >0, x,y \in \R^n$, and
\begin{align}\label{est2}
|K_t(x+h,y)- K_t(x,y)|&\leq \frac{c}{\sqrt{w_t(x) w_t(y)}}
\biggl(\frac{|h|}{t^{{1}/{2}} + |x-y|}\biggr)^{\nu} e^{- \frac{|x-y|^2}{ct}},\notag\\
|K_t(x,y+h)- K_t(x,y)| &\leq \frac{c}{\sqrt{w_t(x) w_t(y)}}
\biggl(\frac{|h|}{t^{{1}/{2}} + |x-y|}\biggr)^{\nu} e^{- \frac{|x-y|^2}{ct}},
\end{align}
for all $t >0, x,y, h \in \R^n$, satisfying $2|h| \leq t^{{1}/{2}} + |x-y|$.
\end{thm}

\begin{rem}\label{rema}
    The constant $\frac{1}{\sqrt{w_t(x) w_t(y)}}$ in Theorem \ref{main} can be changed into one of
    \begin{align*}
        \frac{1}{w_t(x)}, \quad \frac{1}{w_t(y)}, \quad \frac{1}{\max{(w_t(x),w_t(y))}},
    \end{align*}
    if the constant $c$ is replaced with {$\Tilde{c}$ which also depends on the structural constants}, see \cite[Rem. 3]{CUR1}.
\end{rem}

As discussed, in the non-degenerate case  $w\equiv 1$, Theorem \ref{main} is well known, and we refer to  \cite{Ar67,F} for the existence of the fundamental solution. After the groundbreaking work {of Nash} in \cite{Na}, in which certain estimates of the fundamental solutions and  H{\"o}lder continuity of the weak solutions were established, there were several important contributions in the field. As mentioned in \cite{Ar67}, two-sided Gaussian bounds for the fundamental
solutions were proved by employing by now the standard parabolic Harnack inequality. Subsequently,  in \cite{FaSt} it was shown that Nash’s method can also be used to
prove Aronson’s Gaussian bounds.

The quantitative estimates stated in Theorem \ref{main} were proved
in \cite{CUR0, CUR1} assuming in addition that $A$ is symmetric and independent of $t$. We note that there are certain differences between \cite{Ar67,Ar68} and the approach used in \cite{CUR0, CUR1}. Indeed, in contrast to \cite{Ar67,Ar68}, \cite{CUR0, CUR1} employ an argument along the lines of Davies \cite{Davies} to derive Gaussian bounds. The latter argument relies on off-diagonal estimates, the Harnack inequality, and an $\L^{\infty}(\R^n) \to \L^2_w(\R^n)$ bound for the solution operator. Also, for the existence part, in \cite{CUR0, CUR1} the fact that $\mathcal{L}= -{w}^{-1} \div_x (A(x) \nabla_x)$ is induced through the accretive sesquilinear form, $$\int_{\R^n}  A(x) \nabla_x u\cdot \overline{\nabla_x v}  \, \d x,$$  is used. As a consequence, the exponential operator $e^{-t \mathcal{L}}$ is well-defined and  the action of $ e^{-t  \mathcal{L}} $ on $\L^2_w(\R^n)$ induces the fundamental solution. However, this idea does not work if $A$ is time-dependent.

The contribution of this note is that we generalize {the result of Cruz-Uribe and Rios in \cite[Thm. 1.3]{CUR0, CUR1}} to operators with (not necessarily symmetric) time-dependent coefficients. To accomplish this, we have to proceed differently compared to \cite{CUR0, CUR1}, avoiding the use of the exponential operator $e^{-t \mathcal{L}}$, and we do so by first returning to the outstanding
work of Kato \cite{Kato}. In {\cite[Thm. I]{Kato}}, existence and uniqueness of solutions to the initial value problem
for the evolution equation
\begin{align}\label{evo}
\frac {\d u}{\d t}+\mathcal{A}(t)u = f(t),\, 0<t<T,
\end{align}
were studied. Here, the unknown $u = u(t)$ and the inhomogeneous term $f(t)$ are functions from the interval $[0,T]$ to a Banach space $\mathcal{X}$, whereas $\mathcal{A}(t)$ is a function from $[0, T]$ to the set of (in general unbounded) linear operators acting in $\mathcal{X}$.  Given initial data in
$\mathcal{X}$,  in \cite{Kato} the existence and uniqueness of solutions to the abstract Cauchy problem in \eqref{evo} are proved assuming, roughly speaking, that (i) $-\mathcal{A}(t)$ is the infinitesimal generator of an
analytic semigroup of operators; (ii) for some $h=1/m$, where $m$ is a positive integer, the domain of $(\mathcal{A}(t))^h$  is independent of $t$;
(iii) $\mathcal{A}(t)$ varies smoothly
with $t$, see \cite{Kato} and our discussion below.

In particular, to use  \cite[Thm. I, Thm. III]{Kato} and to prove Theorem \ref{main}, we first note that in our case, $\mathcal{A}(t)$ is formally induced through
$$\langle \mathcal{A}(t) u,v\rangle:=\langle \mathcal{L} u,v\rangle=\int_{\R^n}  A(x,t) \nabla_x u\cdot\overline{\nabla_x v} \,  \d x.$$
While $\mathcal{A}(t)$ initially is an unbounded operator on $\L^2_w(\R^n)$, we consider its restriction to
\begin{align}\label{domain}
\dom( \mathcal{A}(t)):= \{u \in \H_w^1({\mathbb{R}}^n) :  \mathcal{A}(t)u \in \L^2_w(\R^n) \}.
\end{align}
Assuming sufficient regularity in $t$, (i) above follows from ellipticity. Furthermore, (ii) with $m=2$ is a consequence of the solution of the Kato problem for degenerate elliptic operators, see \cite{CUR2}. However, if we have sufficient regularity in $t$, then (ii) also follows from \cite{Kato} for some $m\geq 3$ and in this sense the solution of the Kato problem is not needed. Independent of method to conclude (ii), we prove, after an initial regularization of $A$ in the time component and following \cite{Kato},  the existence of a kernel/fundamental solution to certain operators approximating our original operator. We then prove appropriate off-diagonal estimates by following the argument in \cite[Lem.1]{Davies}, and we proceed {as in}
 \cite{CUR0, CUR1} to establish upper  Gaussian bounds. Finally, we remove the regularization parameters and pass them to the limit in a convergence argument.

After some preliminaries, the rest of the paper is devoted to the proof of Theorem \ref{main}.

\section{Preliminaries and basic assumptions}
\label{sec:prel}
For general background and the results concerning weights cited here, we refer to \cite[Ch.~V]{Stein}.  The weight $w=w(x)$ is a real-valued function belonging to the Muckenhoupt class $A_2(\mathbb R^{n},\d x)$, that is,
\begin{equation}
	\label{A2}
[w]_{A_2} \coloneqq \sup_Q \bigg(\barint_{Q} w\, \d x\bigg)\,
\bigg(\barint_{Q} w^{-1} \, \d x \bigg) < \infty,
\end{equation}
where the supremum is taken with respect to all cubes $Q\subset\mathbb R^n$. We introduce the measure $\d w(x):=w(x) \d x$ on $\mathbb R^{n}$, and we write $w(E) := \int_{E} \d w$ for all Lebesgue measurable sets $E \subset \R^n$. It follows from \eqref{A2} that there are constants $\eta \in (0,1)$ and $\beta>0$, depending only on $n$ and $[w]_{A_2}$, such that
\begin{align}\label{ainfw}
	\beta^{-1} \biggl (\frac {|E|}{|Q|}\biggr )^{\frac{1}{2 \eta}} \leq \frac {w(E)}{w(Q)}\leq \beta \biggl (\frac {|E|}{|Q|}\biggr )^{2\eta},
\end{align}
whenever $E\subset Q$ is measurable and where $|\cdot|$ denotes Lebesgue measure in $\R^n$.  In particular, there exists a constant $D$ only depending on $[w]_{A_{2}}$ and $n$,  called the doubling constant for $w$, such that
\begin{align}\label{eq:doubling}
	w(2Q)\leq Dw(Q) \mbox{ for all cubes } Q\subset\mathbb R^n.
\end{align}
Since, by {equation \eqref{A2}}, the function $\frac{1}{w}$ belongs to $A_2(\R^n, \d x)$, \eqref{eq:doubling} holds for $\frac{1}{w}.$
For every $p\geq 1$ and $K \subset \R^n$, the space $\L^p_w(K)$ is the space of  all measurable functions $f: \R^n \to \IC$ such that
\begin{align*}
    \|f\|_{\L^p_w(K)} := \biggl(\int_{K} |f|^p \,  \d w \biggr)^{\frac{1}{p}} < \infty.
 \end{align*}
 We denote $\L^p_w := \L^p_w(\R^n).$

 We define $\langle \, , \, \rangle_w$ as the inner product induced by the norm $\|\,\|_{\L^2_w}$. Using the $A_2$-condition, we have
\begin{align}\label{embedd}
	\L^2_w \subset \Lloc^1(\R^n, \d x),
\end{align}
and the class $\C_0^\infty(\R^n)$ of smooth and compactly supported test functions is dense in $\L^2_w$ via the usual truncation and the convolution procedure~\cite[Sec.~1]{Kilp}. Finally, we  write $\H_w^{1} := \H_w^{1}(\R^n)$ for the space of all $f \in \L_w^2$ for which the distributional gradient $\gradx f$ is (componentwise) in $\L_w^2$, and we equip the space with the norm $$\|\cdot\|_{\H_w^1} \coloneqq (\|\cdot\|_{2,w}^2 + (\|\gradx \cdot\|_{2,w}^2)^{1/2}.$$
By construction $\H^1_w$ is a Hilbert space and the standard truncation and convolution techniques yield that $\C_0^\infty(\R^n)$ is dense in $\H_w^{1}$, see~\cite[Thm.~2.5]{Kilp}. { We also introduce the space $\H^{1,1}_{w,0}( \R^n \times (0,T))$ as the completion of $C_0^{\infty}(\R^n \times (0,T))$ with the norm}
\begin{align*}
    (\|\cdot \|^2_{2,w}+ \|\partial_t \cdot  \|^2_{2,w} + \|\nabla_x \cdot \|^2_{2,w})^{1/2}.
\end{align*}
Given an operator  $\mathcal{L}$ defined on a subset of $\L^2_w$, we introduce
\begin{align*}
    \dom(\mathcal{L}) := \{u \in \L^2_w:\mathcal{L}(u) \in \L^2_w \}.
\end{align*}
 A quadratic form $\Phi: \H^1_w \to \R$ is said to be closed if for every sequence $u_n \in \H^1_w$, satisfying
\begin{align*}
    \lim_{m,n \to \infty}  \Phi[u_m - u_n]&=0\mbox{ and }\lim_{i \to \infty} \|u_n-u\|_{\L^2_w} =0,
\end{align*}
 for some $u \in \L^2_w$, we have  $u \in \H^1_w$ and that
 \begin{align*}
     \lim_{n \to \infty} \Phi[u_n-u] =0.
 \end{align*}
From now on, the notation $A \lesssim B$ means that $A \leq c B$ for some constant $c$, depending at most on the structural constants unless otherwise stated. The notations $A \gtrsim B$ and $A \sim B$ should be interpreted similarly.

\section{Proof of Theorem \ref{main}: uniqueness}

We here prove the uniqueness part of Theorem \ref{main} by proceeding along the lines of the corresponding proof in {\cite[Lemma 1]{Ar68}}. The argument uses arguments to be  found in \cite{Evans} and properties of Steklov averages. However, since the argument uses Steklov averages and contains some subtleties,  full details are supplied.

To prove uniqueness, it is enough to prove that if $u$ is a weak solution to the Cauchy problem such that
 $$u\in \L^\infty([0,T], \L^2_w({\mathbb{R}}^n))\cap \L^2((0,T], \H_w^1({\mathbb{R}}^n)),$$ and $u(\cdot,t)\to 0$ in $ \L^2_w(\mathbb R^n)$ as $t\to 0^+$, then $u=0$ a.e. in $\mathbb R^n\times [0,T]$.  We note that by an approximation argument in $C^{\infty}_0(\R^{n} \times (0,T)),$  test functions in the space $\H^{1,1}_{w,0}(\R^n \times (0,T))$ are allowed in the weak formulation of the problem and $ \partial_t u \in \L^2([0,T],\H^{-1}_w(\R^n)).$ Moreover, by Lemma \ref{lem:continuityintime} below, we can assume that $u \in C([0,T],\L^2_w(\R^n))$ after redefining it on a set of measure zero.

 To proceed, we fix $T' \in (0,T)$ 
and consider the test function $$\phi(x,t) := \zeta_h(t) \barint_{t-h}^{t+h} u(x,s)\, \d s,$$ where
 \begin{align*}
     \zeta_h(t) := \begin{cases}
     0, \quad & t \in [0,h],\\
 \frac{t-h}{h}, \quad & t \in (h,2h],\\
         1, \quad &t \in (2h,T'-2h],\\
           \frac{T'-h-t}{h} , \quad &t \in (T'-2h,T'-h],\\
          0, \quad & t \in (T'-h,T].
     \end{cases}
 \end{align*}
 Using the equation for $u$,
 \begin{equation}
 \label{eq:testequation}
 \begin{aligned}
    &\int_0^T \int_{\R^n} u(x,t) \partial_t \left(\zeta_h(t) \barint_{t-h}^{t+h} u(x,s)\, \d s\right) \,  \d w \d t
    \\&= \int_0^T \int_{\R^n} A(x,t) \nabla_x u(x,t) \cdot \nabla_x \left(\zeta_h(t) \barint_{t-h}^{t+h} u(x,s)\, \d s\right)\,  \d x \d t.
 \end{aligned}
 \end{equation}
Using $u\in C([0,T],\L^2_w(\R^n))$ and Lebesgue's dominated convergence, we obtain
\begin{equation}
\label{eq:convergencetoL2}
\begin{aligned}
& \lim_{h \to 0}   \int_0^T \int_{\R^n} u(x,t) \left( \barint_{t-h}^{t+h} u(x,s)\, \d s\right) \partial_t\zeta_h(t)   \,  \d w \d t
 \\ &= \int_{\R^n} u^2(x,0) \, \d w - \int_{\R^n} u^2(x,T') \, \d w,
\end{aligned}
\end{equation}
and
\begin{equation}
\label{eq:convergencetoH^1}
\begin{aligned}
& \lim_{h \to 0}    \int_0^T \int_{\R^n} A(x,t) \nabla_x u(x,t) \cdot \nabla_x \left(\zeta_h(t) \barint_{t-h}^{t+h} u(x,s)\, \d s\right)\,  \d x \d t
 \\ &=  \int_0^{T'} \int_{\R^n} A(x,t) \nabla_x u(x,t) \cdot \nabla_x u(x,t) \, \d x \d t.
\end{aligned}
\end{equation}
Furthermore,
\begin{equation*}
\begin{aligned}
    &\int_0^T \int_{\R^n} u(x,t) \zeta_h(t) \partial_t \left( \barint_{t-h}^{t+h} u(x,s)\, \d s\right) \,  \d w \d t \\&= \frac{1}{2h} \int_0^T \int_{\R^n} u(x,t) \zeta_h(t) ( u(x,t+h) - u(x,t-h)) \,  \d w \d t
    \\ &= \frac{1}{2h} \int_0^T \int_{\R^n} u(x,t) u(x,t+h) (\zeta_h(t)-\zeta_h(t+h))  \,  \d w \d t.
    \end{aligned}
\end{equation*}
We write
\begin{align*}
  & \frac{1}{2h} \int_0^T \int_{\R^n} u(x,t) u(x,t+h) (\zeta_h(t)-\zeta_h(t+h))  \,  \d w \d t=:J_1+J_2 +J_3+ J_4,
\end{align*}
where
\begin{align*}
   J_1&:= -\frac{1}{2h} \int_0^h \frac{t}{h} \int_{\R^n} u(x,t) u(x,t+h)  \,  \d w \d t,
   \\ J_2&:=- \frac{1}{2h} \int_h^{2h} \left(1- \frac{t-h}{h} \right) \int_{\R^n} u(x,t) u(x,t+h)   \,  \d w \d t,
   \\J_3&:=-\frac{1}{2h} \int_{T'-3h}^{T'-2h} \left(\frac{T'-2h-t}{h}-1\right) \int_{\R^n} u(x,t) u(x,t+h)   \,  \d w \d t,
   \\ J_4&:=\frac{1}{2h} \int_{T'-2h}^{T'-h} \frac{T'-h-t}{h} \int_{\R^n} u(x,t) u(x,t+h)   \,  \d w \d t.
\end{align*}
We will establish the convergence of each term on the right-hand side as $h\to 0$. First, by Cauchy-Schwarz inequality, we obtain
\begin{align*}
    &\left|J_1 + \frac{1}{2h} \int_0^h \frac{t}{h} \int_{\R^n} u^2(x,0) \, \d w  \d t\right |\\
   & \leq \frac{1}{2h} \int_0^h \frac{t}{h} \int_{\R^n}  |u(x,t) - u(x,0)| \,|u(x,t+h)| \d w \d t
   \\& + \frac{1}{2h} \int_0^h \frac{t}{h} \int_{\R^n}  |u(x,0)| \,|u(x,t+h)-u(x,0)| \, \d w \d t
   \\& \leq \frac{1}{2} \left(\sup_{0 \leq t \leq T} \int_{\R^n} |u(x,t)|^2 \, \d w \right)^{\frac{1}{2}} \left(\barint_0^h  \int_{\R^n}  |u(x,t) - u(x,0)|^2 \d w \d t \right)^{\frac{1}{2}}
   \\ &+\frac{1}{2} \left(\int_{\R^n} |u(x,0)|^2 \, \d w \right)^{\frac{1}{2}} \left(\barint_h^{2h}  \int_{\R^n}  |u(x,t) - u(x,0)|^2 \d w \d t \right)^{\frac{1}{2}}.
\end{align*}
Letting $h \to 0$, we derive that
\begin{align*}
   \lim_{h \to 0} J_1 = -\frac{1}{4} \int_{\R^n} u^2(x,0) \, \d w.
\end{align*}
Similarly,
\begin{align*}
    \lim_{h \to 0} J_2 = -\frac{1}{4} \int_{\R^n} u^2(x,0) \, \d w.
\end{align*}
For the term $J_3$, by Cauchy-Schwarz inequality, we have
\begin{align*}
    &\left| J_3 + \frac{1}{2h} \int_{T'-3h}^{T'-2h} \left(\frac{T'-2h-t}{h}-1\right) \int_{\R^n} u^2(x,T')   \,  \d w \d t  \right|
    \\ &\leq  \frac{1}{2h} \int_0^h  \int_{\R^n}  |u(x,t) - u(x,T')| \,|u(x,t+h)| \d w \d t
   \\& + \frac{1}{2h} \int_0^h \int_{\R^n}  |u(x,T')| \,|u(x,t+h)-u(x,T')| \, \d w \d t
   \\& \leq \frac{1}{2} \left(\sup_{0 \leq t \leq T} \int_{\R^n} |u(x,t)|^2 \, \d w \right)^{\frac{1}{2}} \left(\barint_0^h  \int_{\R^n}  |u(x,t) - u(x,T')|^2 \d w \d t \right)^{\frac{1}{2}}
   \\ &+ \frac{1}{2} \left(\int_{\R^n} |u(x,T')|^2 \, \d w \right)^{\frac{1}{2}} \left(\barint_h^{2h}  \int_{\R^n}  |u(x,t) - u(x,T')|^2 \d w \d t \right)^{\frac{1}{2}}.
\end{align*}
Hence, letting $h \to 0$ and using $u\in C([0,T],\L^2_w(\R^n))$,
we derive that
\begin{align*}
    \lim_{h \to 0} J_3 = \frac{1}{4} \int_{\R^n} u^2(x,T') \, \d w.
\end{align*}
Similarly,
\begin{align*}
     \lim_{h \to 0} J_4 = \frac{1}{4} \int_{\R^n} u^2(x,T') \, \d w.
\end{align*}
Put together, we can conclude that
\begin{align}
    \label{eq:thecriticalterm}
&\lim_{h \to 0}   \frac{1}{h} \int_0^T \int_{\R^n} u(x,t) u(x,t+h) (\zeta_h(t)-\zeta_h(t+h))  \,  \d w \d t \notag\\
&=\frac{1}{2}  \int_{\R^n} u^2(x,T') \, \d w - \frac{1}{2} \int_{\R^n} u^2(x,0) \, \d w.
\end{align}
Combining \eqref{eq:testequation}, \eqref{eq:convergencetoL2}, \eqref{eq:convergencetoH^1}, and \eqref{eq:thecriticalterm}, we see that
\begin{align*}
 & \frac{1}{2}  \int_{\R^n} u^2(x,0) \, \d w - \frac{1}{2} \int_{\R^n} u^2(x,T') \, \d w\\
   &=\int_0^{T'} \int_{\R^n} A(x,t) \nabla_x u(x,t) \cdot \nabla_x u(x,t) \, \d x \d t\geq 0.
\end{align*}
Hence, using that $u(x,0)=0$, we see that $u=0$ in $[0,T'] \times \R^n$. This completes the proof.

\begin{lem}
\label{lem:continuityintime}
    Let $T>0, u \in \L^2([0,T],\H^1_w(\R^n))$, and $\partial_t u \in \L^2([0,T],\H^{-1}_w(\R^n))$. Then, $u \in C([0,T], \L^2_w(\R^n))$ after being redefined on a set of measure zero.
\end{lem}
\begin{proof}
    The proof follows the argument in \cite[Thm. 3, Ch. 5.9]{Evans}. First, we extend the function $u$ in a larger time interval $[-\theta,T+\theta]$ for $\theta>0$ and define a regularization $u^\varepsilon=\eta_{\varepsilon} \ast u$ where $\eta_{\varepsilon} \in C^{\infty}_0(\R)$ is a smooth approximation of the identity. For $\varepsilon, \delta >0$ and $0 \leq s \leq r \leq T$, we have
    \begin{align*}
   & \|u^\varepsilon(\cdot,r) - u^\delta(\cdot,r)\|^2_{\L^2_w} =\|u^\varepsilon(\cdot,s) - u^\delta(\cdot,s)\|^2_{\L^2_w}  + \int_s^r \frac{d}{d t}  \|u^\varepsilon - u^\delta\|^2_{\L^2_w} \, \d t
    \\ &\leq \|u^\varepsilon(\cdot,s) - u^\delta(\cdot,s)\|^2_{\L^2_w} + 2\left(\int_s^r \|u^\varepsilon - u^\delta\|^2_{\H^1_w} \, \d t\right)^{\frac{1}{2}} \left(\int_s^r \|\partial_t u^\varepsilon - \partial_t u^\delta\|^2_{\H^{-1}_w} \, \d t\right)^{\frac{1}{2}},
    \end{align*}
 where we used Cauchy-Schwarz inequality on the last line. Now, note that
\begin{align*}
 \int_s^r \|\partial_t u^\varepsilon - \partial_t u^\delta\|^2_{\H^{-1}_w} \, \d t&=\int_s^t\left(\sup_{\phi} \int_{\R^n}  \partial_t \left(u^\varepsilon - u^\delta\right) \, \phi\, \d w\right)^2 \, \d t
 \\&\leq \int_s^r \left(\left|\eta_{\varepsilon}-\eta_{\delta}\right| \ast \|\partial_t u(t,\cdot)\|_{\H^{-1}_w} \right)^2 \d t
 \\ & \leq \|\partial_t u\|^2_{\L^2([0,T];\H^{-1}_w(\R^n))} \left(\int_{\R} |\eta_{\varepsilon}-\eta_{\delta}| \, \d t\right)^{2}
\end{align*}
where the supremum is take over $\phi \in \H^1_w(\R^n)$ where $\|\phi\|_{\H^1_w(\R^n)}=1$, and we used Young's convolution inequality on the last line.
Hence, by fixing $0<s<T$ such that $$\lim_{\varepsilon \to 0}\|u^{\varepsilon}(\cdot,s)-u(\cdot,s)\|_{\L^2_w}=0,$$ we conclude that
\begin{align*}
    \limsup_{\varepsilon \to 0, \delta \to 0}  \sup_{0\leq r \leq T}\|u^\varepsilon(\cdot,r) - u^\delta(\cdot,r)\|^2_{\L^2_w} =0.
\end{align*}
In conclusion, $u^\varepsilon$ converges in $C([0,T],\L^2_w(\R^n))$ to a function $v \in C([0,T],\L^2_w(\R^n))$ as $\varepsilon \to 0.$ Since $u^\varepsilon(t,\cdot)$ converges to $u(t,\cdot)$ in $\L^2_w(\R^n)$ for a.e. $t \in [0,T],$ we conclude that $u=v$ a.e.
\end{proof}

\section{Proof of Theorem \ref{main}: existence and kernel representation}
We here prove the existence part of Theorem \ref{main} and the stated representation in terms of a kernel. Our first step is to use \cite[Thm. III]{Kato}, and to do so we in particular have to work with coefficients which are smooth in the time variable. Hence, we have to prove uniform estimates for a class of approximating operators and then pass to the limit. We divide the argument into a number of relevant steps.

\subsection{Existence of linear evolution operators following Kato} Let $\rho\in C^\infty_0(-1,1)$ be a non-negative function which integrates to $1$. Given $l\in\mathbb R_+$ and $\rho_l(t) = {l\rho(lt)},$
we introduce $A_l(\cdot,t) = \rho_l \ast A(\cdot,t)$, i.e., we mollify the matrix-valued function $A$ in the time variable only. Define the sesquilinear form
\begin{align*}
    \Phi^l(t)(u,v) := \int_{\R^n} w^{-1}A_l(x,t) \nabla_x u \cdot \overline{\nabla_x v} \, \d w +  {l}^{-1} \int_{\R^n} u  \, \overline{v} \, \d w,
\end{align*}
 for every $u,v \in \H^1_w$ and $\mathcal{L}_l^t$ through
 $$\langle \mathcal{L}_l^tu,v\rangle_w:= \Phi^l(t)(u,v).$$
Formally,
 $$\mathcal{L}_l^t= -w^{-1}\div_x (A_l(x,t) \nabla_{x}) + {1}/{l}.$$
In $\mathcal{L}_l^t$ and $\Phi^l(t)$, $t$ should be seen as a parameter.

 Let $\Phi^l(t)[u] := \Phi^l(t)(u,u)$ for every $u \in \H^1_w$. Then,
 \begin{align}
 \label{eq:quadraticform}
    \Im \Phi^l(t)[u] \leq \frac{c_2}{c_1} \Re \Phi^l(t)[u], \quad \Re  \Phi^l(t)[u] \geq \min\{c_1,1/l\}\|u\|^2_{\H^1_w},
 \end{align}
 for every $t \in \R, u \in \H^1_w.$
 Let $u_n \in \H^1_w$  be a sequence such that $$\lim_{n \to \infty} \|u_n-u\|_{\L^2_w}=0,$$
 for $u \in \L^2_w$, and $$\lim_{m,n \to \infty} \Re  \Phi^l(t)[u_m - u_n]=0.$$ Then, by \eqref{eq:quadraticform}, $u_n$ is a Cauchy sequence in the Hilbert space $\H^1_w$. Hence, $$\lim_{n \to \infty} \|u_n-u\|_{\H^1_w}= 0,$$ and $$\lim_{n \to \infty} \Re  \Phi^l(t)[u_n] = \Phi^l(t)[u].$$ This proves that $\Re \Phi^l(t)$ is a closed quadratic form.
 Now,
 \begin{align*}
     |\Phi^l(t)[u] - \Phi^l(s)[u]| &= \biggl|\int_{\R^n} (A_l(x,t)- A_l(x,s)) \nabla_x u \cdot \overline{\nabla_x u} \, \d x \biggr|.
 \end{align*}
 for all $s,t \in \R, u \in \H^1_w.$  Noting that
 \begin{align*}
 w^{-1}(x)(A_l(x,t)- A_l(x,s))=\int w^{-1}(x)A(x,\tau)(\rho_l(\tau-t)-\rho_l(\tau-s))\, \d\tau,
 \end{align*}
 we deduce that
  \begin{align*}
 |w^{-1}(x)(A_l(x,t)- A_l(x,s))|\lesssim \int |\rho_l(\tau-t)-\rho_l(\tau-s)|\, \d\tau\lesssim l{\|\partial_t \rho \|_{\L^{\infty}}}|t-s|,
 \end{align*}
 for all $s,t \in \R$, where the second implicit constant also depends on $\rho$. Hence, {by \eqref{eq:quadraticform}, we have}
 \begin{align*}
     |\Phi^l(t)[u] - \Phi^l(s)[u]|  & \lesssim l{\|\partial_t \rho \|_{\L^{\infty}}} |t-s| \|\nabla_x u\|^2_{\L_w^2}\lesssim l{\|\partial_t \rho \|_{\L^{\infty}}} |t-s| \Re \Phi^l(s)[u],
 \end{align*}
for all $s,t \in \R, u \in \H^1_w.$  Now applying  {\cite[Thm. III]{Kato}},
we can conclude the following.
\begin{thm}
    \label{fundamental} For every $T>0$, there exists a
unique bounded linear evolution operator  $U_l(t,s): \L^2_w \to \L^2_w$, defined for $0 \leq s \leq t\leq T$, with the following properties: \\

\noindent 1. $U_l(t,s)$ is strongly continuous for $0 \leq s \leq t\leq T$ and
\begin{align*}
    \mathrm{(i)}&\mbox{ $U_l(t,t) = 1$, for all $t \geq 0$}, \\
    \mathrm{(ii)}&\mbox{ $U_l(t,s) U_l(s,r) = U_l(t,r),$ for all $0 \leq r \leq s$.}
\end{align*}

\noindent 2. For $0 \leq s<t$, the range of $U_l(t,s)$ is a subset of $\dom( \mathcal{L}_l^t)$, $\mathcal{L}_l^t U_l(t,s):\L^2_w \to \L^2_w $ is a bounded operator, $U_l(t,s)$ is strongly differentiable in $t$, and
\begin{align*}
    \mathrm{(iii)}&\ \partial_t U_l(t,s) f + \mathcal{L}_l^t U_l(t,s) f =0,\mbox{ for all $f \in \L^2_w.$}
     \end{align*}
\end{thm}

For simplicity, we will write $\mathcal{L}_l$ instead of $\mathcal{L}_l^t$, hence suppressing the superscript $t$. We will need the following result.

\begin{lem}
\label{lem:sign}
 If $f \in \L^2_w$ is a real-valued non-negative function, then $U_l(t,0)f$ is also real-valued and non-negative for all $ t \geq 0$.
\end{lem}
\begin{proof}
By property (i) of Theorem \ref{fundamental}, the lemma is immediate for $t=0.$ Let $f \in \L^2_w$ be a real-valued non-negative function and consider $ t >0$.  Using the inequality
     \begin{align*}
   0 \leq & \Re \langle \mathcal{L}_l U_l(t,0)f , U_l(t,0)f - \Re U_l(t,0) f \rangle_{w},
\end{align*}
we have
\begin{align*}
   0 \leq & \Re \langle \mathcal{L}_l U_l(t,0)f , U_l(t,0)f - \Re U_l(t,0) f \rangle_{w}\\
    &=- \Re \langle \partial_t U_l(t,0)f, U_l(t,0)f - \Re U_l(t,0) f  \rangle_w \\
    &= -\frac{1}{2} \partial_t \langle  U_l(t,0)f, U_l(t,0)f \rangle_w + \frac{1}{2} \partial_t \langle \Re U_l(t,0) f, \Re U_l(t,0) f \rangle_w.
\end{align*}
Integrating from $0$ to $t$ in this inequality, we have
\begin{align*}
    \langle  U_l(t,0)f, U_l(t,0)f \rangle_w \leq  \langle \Re U_l(t,0) f, \Re U_l(t,0) f \rangle_w.
\end{align*}
In conclusion, $\Im{U_l(t,0)f}=0$ and $U_l(t,0)f$ is a real-valued function. Since both $\mathcal{L}_l U_l(t,0)f$ and $f$ belong to $\L^2_w$, we deduce that
\begin{align*}
    \| \nabla_{x} U_l(t,0)f \|_{\L^2_w} \lesssim \langle {\mathcal{L}_l} U_l(t,0)f, U_l(t,0) f \rangle_w < \infty,
\end{align*}
and that $\partial_t U_l(t,0) f \in \L^2_w $. By a standard argument, $\partial_t |U_l(t,0) f|, \nabla_{x} |U_l(t,0) f| \in \L^2_w$ and
\begin{equation*}
    (\partial_t |U_l(t,0) f|, \nabla_{x}|U_l(t,0) f|) = \begin{cases}
        (\partial_t U_l(t,0) f, \nabla_{x} U_l(t,0) f) \quad &\textup{ if } U_l(t,0) f \geq 0,\\
        (-\partial_t U_l(t,0) f, -\nabla_{x} U_l(t,0) f) \quad &\textup{ if } U_l(t,0) f < 0.
    \end{cases}
\end{equation*}
Using this, we deduce
\begin{align*}
   0 &\leq \Re \langle \mathcal{L}_l U_l(t,0)f , U_l(t,0)f -  |U_l(t,0) f| \rangle_w\\
    &= \Re \langle -\partial_t U_l(t,0)f , U_l(t,0)f -  |U_l(t,0) f| \rangle_w \\ &= -  \langle \partial_t (U_l(t,0)f -  |U_l(t,0) f|) , U_l(t,0)f -  |U_l(t,0) f| \rangle_w \\ &= -\frac{1}{2} \partial_t \langle U_l(t,0)f -  |U_l(t,0) f| , U_l(t,0)f -  |U_l(t,0) f| \rangle_w.
\end{align*}
Integrating from $0$ to $t$ in this inequality, we have
$U_l(t,0)f =  |U_l(t,0) f|$ and hence $U_l(t,0)f $ is non-negative. \end{proof}

\subsection{An off-diagonal estimate and its implications}  Given two closed subsets $E,F  \subset \R^n$, we let $\dist(E,F)$ denote the Euclidean distance between the sets.
\begin{lem}
\label{off}
    Let $E,F  \subset \R^n$ be two closed subsets and let $d:= \dist(E,F)$.  Then, there exists a constant $c>0$, depending only on the structural constants, such that
    \begin{align*}
\|U_l(t,0) (f 1_E) \|_{\L^2_w(F)} \lesssim e^{\big(-\frac{c d^2}{t}\big)} \|f\|_{\L^2_w(E)},
    \end{align*}
for every $t >0$ and for all $f \in \L^2_w(E)$.

\end{lem}
\begin{proof}
    The argument is similar to \cite[Lem. 1]{Davies}. Let $\psi(x):= \dist(x,F)$ and $\phi(x):= e^{\alpha \psi(x)}$, where $\alpha$ is a negative constant to be determined later. Then, by Young's inequality for products, and the fact that $\|\nabla_{x} \psi\|_{\L^{\infty}} \leq 1$, we have
\begin{equation*}
\begin{aligned}
    \partial_t \| \phi U_l(t,0) (f 1_E) \|^2_{\L^2_w} & = -2 \langle \mathcal{L}_l U_l(t,0)(f 1_E), \phi^2 U_l(t,0)(f 1_E)   \rangle_w  \\  \leq& -2 \langle A_l \nabla_{x} (U_l(t,0) (f 1_E)), \nabla_{x} (\phi^2 U_l(t,0)(f 1_E)) \rangle_w \\ \leq& -2c_1 \|\phi \nabla_{x} (U_l(t,0)(f 1_E)) \|_{\L^2_w}^2 + \frac{2c_2}{\lambda} \|\phi \nabla_{x} (U_l(t,0) (f 1_E))\|^2_{\L^2_w}\\& + \lambda 2 \alpha^2 c_2 \| \phi U_l(t,0) (f 1_E)\|^2_{\L^2_w},
\end{aligned}
\end{equation*}
where $\lambda>0$ is a degree of freedom.  Letting $\lambda = {c_2}/{c_1}$, we obtain
\begin{align*}
    \partial_t \| \phi U_l(t,0) (f 1_E) \|_{\L^2_w}^2 \leq \frac{2 \alpha^2 c_2^2 }{c_1} \|\phi U_l(t,0) (f 1_E)  \|_{\L^2_w}^2.
\end{align*}
Hence,
\begin{align*}
    \| \phi U_l(t,0) (f 1_E) \|_{\L^2_w}^2 \leq e^{\big(\frac{2\alpha^2 c_2^2 t}{c_1}\big)} \| \phi f 1_E \|^2_{\L^2_w}.
\end{align*}
    In conclusion,
\begin{align*}
    \int_{F} |U_l(t,0) (f 1_E)|^2 \, \d w  &\leq   \int_{\R^n} |U_l(t,0) (f 1_E)|^2 \phi^2 \, \d w\\
     &\leq e^{\big(\frac{2\alpha^2 c_2^2 t}{c_1}\big)} \| \phi f 1_E \|^2_{\L^2_w} \\ & \lesssim  e^{\big(\frac{2\alpha^2 c_2^2 t}{c_1} + 2 \alpha d \big)} \| f 1_E \|^2_{\L^2_w}.
\end{align*}
We conclude the proof by letting $\alpha = -{(dc_1)}/{(2 c_2^2 t)}.$ \end{proof}

We introduce the cylinders
\begin{align*}
    C_{r}(x_0,t_0) :=& \biggl\{(x,t): |t-t_0| < r^2, |x-x_0| < 2r\biggr\}, \\
       C_{r}^+(x_0,t_0) :=& \biggl\{(x,t): 3r^2/4 <t-t_0 < r^2, |x-x_0| <  {r}/{2}\biggr\}, \\
          C_{r}^-(x_0,t_0) :=& \biggl\{(x,t): - 3r^2/4 < t-t_0 < - r^2/4, |x-x_0| <  {r}/{2}\biggr\},
\end{align*}
for all $r>0$ $(x_0,t_0) \in \R^n\times \mathbb R$. We refer to {\cite[Thm. 2.1]{CS}}, for autonomous coefficients, and \cite[Thm. A]{Ishige} for the proof of the following
Harnack inequality.

\begin{lem}
\label{Harnack}
Let $(x_0,t_0) \in \R^n \times \R$, $r>0$. If $u$ is a non-negative weak solution of $\cH u=0$ in $Q_r(x_0,t_0)$, then
\begin{align*}
    \sup_{  C_{r}^-(x_0,t_0)} u(x,t) \lesssim \inf_{C_r^+(x_0,t_0)} u(x,t).
\end{align*}
\end{lem}

 To use the argument of Davies {\cite[Thm. 3]{Davies}} to prove the upper Gaussian bound, we prove the following estimate.
\begin{lem} \label{lem:Davieses}
Let $\phi \in C^{\infty}_0(\R^n)$ and  $\rho:= \|\nabla_{x} \phi\|_{\L^{\infty}}$. Then,
\begin{align}
\label{Davieseq}
    \|\sqrt{w_t} e^{-\phi} U_l(t,{r}) (e^{\phi} f)\|_{\L^{\infty}} \lesssim e^{\alpha {(t-r)} \rho^2} \|f\|_{\L^2_w},\ 0 \leq r<t,
\end{align}
  for all real-valued functions $f \in \L^2_w$, where $\alpha>0$ is a constant, depending on the structural constants.
\end{lem}
\begin{proof} To prove the lemma, we proceed along the line of  \cite[Sec. 5.1]{CUR0, CUR1}, using the previous lemmas. First, by the linearity of $U_l(t,0)$, it is enough to consider the case that $f$ is non-negative. Second, by homogeneity, it suffices to prove that
\begin{align}
\label{maineq}
    | e^{-\phi} U_l(1,0) (e^{\phi} f)(0)| \lesssim e^{\alpha \rho^2} \|f\|_{\L^2_w}.
\end{align}
Indeed, assume that \eqref{maineq} holds for every non-negative function $f \in \L^2_w$, and consider the functions $u(x,t) := e^{-\phi} U_l(t,{r})(e^{\phi} f)(x)$. Now, we consider $t,r>0$ as fixed parameters and let $$v^{t,r}(y,s):= u({x_0+ \sqrt{t-r} \, y}, {r+(t-r)s}),$$ for  $x_0 \in \R^n$ fixed and for all $y \in \R^n, s \in \R_+.$ For $t, r >0$ fixed, we have that $\partial_s v^{t,r}(y,s) $ equals
\begin{align*}
{ e^{-\phi} \biggl(-\frac{1}{w}\div_x A_l(\cdot,r+(t-r)s) \nabla_x +\frac{t-r}{l}\biggr)U(r+(t-r)s,r)(e^{\phi} f)({x_0+ \sqrt{t-r} \, y}),}
\end{align*}
and $v^{t,r}(y,0) = f({x_0+ \sqrt{t-r} \, y}) $ for all $y \in \R^n.$  Hence, $${v^{t,r}}(y,s) = e^{-\phi^{t,r}} U_l^{t,r}(s,0) e^{\phi^{t,r}} f^{t,r}(y), \quad \text{for } y \in \R^n,$$ by the property of uniqueness, where  $$f^{t,r}(y):= f({x_0+ \sqrt{t-r} \, y}),\  \phi^{t,r}(y) := \phi({x_0+ \sqrt{t-r} \, y}), \quad \text{for } y \in \R^n.$$
Furthermore, $U_l^{t,r}(s,0)$ is as in Theorem \ref{fundamental} but induced by the operator
$$
- (w^{t,r})^{-1}\div_{x} (A_l^{t,r} \nabla_{x})+ {\frac{t-r}{l}},$$ where $$A_l^{t,r}(y,s) := A_l({x_0+ \sqrt{t-r} \, y}, {r+(t-r) s}),\  w^{t,r}(y) := w({x_0+ \sqrt{t-r} \, y}), \quad \text{for } y \in \R^n.$$ Since {$A_l^{t,r}$ satisfies
\begin{align*}
    c_1|\xi|^2w^{t,r}(y)\leq A_l^{t,r}(y,s) \xi \cdot {\xi}, \qquad
 |A_l^{t,r}(y,s)\xi\cdot\zeta|\leq c_2w^{t,r}(y)|\xi| |\zeta|,
\end{align*}
for all $y,\xi, \zeta \in \R^n, s\in \R_+$,} $w^{t,r}$ is an $A_2$-weight, and $[w^{t,r}]_{2} = [w]_2,$ the result stated in the lemma is now implied by applying \eqref{maineq} to the function $v^{t,r}$.

Finally, we prove \eqref{maineq}. To start the argument, let $f \in \L^2_{w}$ be a fixed non-negative function and let $Q_0 \subset \R^n$ be the cube centered at the origin with $\ell(Q_0)=9$. We let $Q_k := 3^k Q_0$, and, for $k \geq 1$, $\{Q^{k,j}\}_{j=1}^{3^n-1}$ be a partition of $Q_{k} \setminus Q_{k-1}$ into cubes of side-length $3^{k+1}$. Define $f^0 := f 1_{Q_0}$ and $f^{k,j}:= f 1_{Q^{k,j}}.$ Then,
\begin{align}
\label{trian}
    | e^{-\phi} U_l(1,0) (e^{\phi} f)(0)| &\leq \sum_{k=1}^{\infty} \sum_{j=1}^{3^n-1}  | e^{-\phi} U_l(1,0) (e^{\phi} f^{k,j})(0)| \notag\\
    &+\sum_{j=1}^{3^n-1}  | e^{-\phi} U_l(1,0) (e^{\phi} f^0)(0)|.
\end{align}
Let $u^{k,j}(x,t):= U_l(t,0) (e^{\phi} f^{k,j})(x)$ and $k \geq 1$. Then, by Lemma \ref{lem:sign}, $u^{k,j}$ is a non-negative weak solution of $\partial_t u + \mathcal{L}_l u=0$. For $y \in \R^n, s \in \R_+$, define the function $v^{k,j}(y,s):= u^{k,j}(3^{k} y,s)$ which satisfies $\partial_t v^{k,j}+\Tilde{\mathcal{L}}^k_l v^{k,j}  =0$ where $$\Tilde{\mathcal{L}}_l^k := -(w^{k})^{-1}\div_x(A_l^k \nabla_{x}) + {1}/{l},$$ and $A_{l}^k(y,s) := A_l(3^k y,s)$, $w^k(y) := w(3^k y).$ Then, by Lemma \ref{Harnack},
\begin{align*}
    \sup_{Q_1^{-}(0, \frac{13}{8})} v^{k,j}(y,s) \lesssim  \inf_{Q_1^{+}(0, \frac{13}{8})} v^{k,j}(y,s).
\end{align*}
Hence,
\begin{align*}
    v^{k,j}(0,1) \lesssim \biggl(w^k(B_{\frac{1}{2}}(0))\biggr)^{-\frac{1}{2}} \biggl(\int_{\frac{19}{8}}^{\frac{21}{8}} \int_{B_{\frac{1}{2}}(0)} | v^{k,j}(y,s)|^2 \, \d w^k(y) \d s  \biggr)^{\frac{1}{2}}.
\end{align*}
By change of variable, this implies that
\begin{align*}
       u^{k,j}(0,1) \lesssim \biggl(w_{{\frac{3^k}{2}}}(0)\biggr)^{-\frac{1}{2}}\biggl(\int_{\frac{19}{8}}^{\frac{21}{8}} \int_{B_{\frac{1}{2}}(0)} | v^{k,j}(y,s)|^2 \, \d w^k(y) \d s  \biggr)^{\frac{1}{2}}.
\end{align*}
Now, $e^{\phi} f^{k,j}$ is supported in $Q^{k,j}$ and $\dist\bigl(Q^{k,j},B_{\frac{3^k}{2}}(0)\bigr) \geq \frac{3^k}{2}$. Hence, by Lemma \ref{off},
\begin{align}\label{eq1}
    & |u^{k,j}(0,1)|\notag\\ \lesssim & \biggl(w_{{\frac{3^k}{2}}}(0)\biggr)^{-\frac{1}{2}} e^{-c3^{2k}}\biggl(\int_{\frac{19}{8}}^{\frac{21}{8}} \int_{Q^{k,j}} e^{2(\phi(x)-\phi(0))} | f^{k,j}(y,s)|^2 \, \d w^k(y) \d s  \biggr)^{\frac{1}{2}}\notag
    \\  \lesssim &\biggl(w_{{\frac{3^k}{2}}}(0)\biggr)^{-\frac{1}{2}} e^{\big(-c3^{2k}+3^{k+1} \frac{\sqrt{n}}{2} \rho\big)} \|f^{k,j}\|_{\L^2_w} .
\end{align}
By Lemma \ref{off} and a similar estimate as above, we obtain
\begin{align}
\label{eq2}
    e^{-\phi(0)} |U_l(1,0) (e^{-\phi} f^0)(0)| \lesssim e^{\big(9\frac{ \sqrt{n} }{2} \rho\big)} \|f^0\|_{\L^2_w}.
\end{align}
Now, by summing \eqref{trian}, \eqref{eq1}, and \eqref{eq2}, we see that
\begin{align*}
   &e^{-\phi(0)} |U_l(1,0)(e^{\phi} f)(0)| \\ \lesssim & \biggl(e^{9 \sqrt{n}  \rho} + \sum_{k=1}^{\infty} \sum_{j=1}^{3^n-1} \biggl(w_{{\frac{3^k}{2}}}(0)\biggr)^{-1} e^{\big(-2 c3^{2k}+3^{k+1} \sqrt{n}\rho\big)}\biggr)^{\frac{1}{2}} \biggl(\|f^{0}\|^2_{\L^2_w} + \sum_{k=1}^{\infty} \sum_{j=1}^{3^n-1} \|f^{k,j}\|^2_{\L^2_w}\biggr)^{\frac{1}{2}}\\
    &{\leq \biggl(e^{9 c} e^{\frac{9n}{c}\rho^2}  + 3^{n} e^{\frac{9n}{c}\rho^2}\sum_{k=1}^{\infty}  \biggl(w_{{\frac{3^k}{2}}}(0)\biggr)^{-1} e^{\big(- c3^{2k}\big)}\biggr)^{\frac{1}{2}} \|f\|_{\L^2_w} }  \\ \lesssim
    & e^{\alpha \rho^2} \|f\|_{\L^2_w},
\end{align*}
where $\alpha$ depends on the structural constants. In the inequalities above, Cauchy-Schwarz inequality is used on the first inequality, and \eqref{eq:doubling} is used on the last inequality. This completes the proof of \eqref{maineq}.\end{proof}

\subsection{Kernel estimates for the operator $U_l(t,0)$} We here prove the Gaussian upper bound estimates for $U_l$.
\begin{thm}
\label{Gaussian}
There exists a kernel $K_t^l(x,y)$ associated with the operator $U_l(t,0)$ such that
\begin{align}
    U_l(t,0)(f) (x) = \int_{\R^n} K_t^l(x,y) f(y) \, \d w(y),
\end{align}
for all $f \in \L^2_w(\R^n)$ and $x \in \R^n$. Furthermore, there exist a constant $c$, $1\leq c<\infty$, and $\nu>0$, both depending only on the structural constants, such that
\begin{equation}\label{Gaussineq}
 K_t^l(x,y) \lesssim\frac{c}{\sqrt{w_t(x) w_t(y)}}
e^{-\frac{|x-y|^2}{ct}},
\end{equation}
for all $t >0, x,y \in \R^n$, and such that
\begin{align}\label{Gaussineq+}
|K_t^l(x+h,y)- K_t^l(x,y)|&\lesssim \frac{1}{\sqrt{w_t(x) w_t(y)}}
\biggl(\frac{|h|}{t^{{1}/{2}} + |x-y|}\biggr)^{\nu} e^{- \frac{|x-y|^2}{ct}},\notag\\
|K_t^l(x,y+h)- K_t^l(x,y)| &\lesssim \frac{1}{\sqrt{w_t(x) w_t(y)}}
\biggl(\frac{|h|}{t^{{1}/{2}} + |x-y|}\biggr)^{\nu} e^{- \frac{|x-y|^2}{ct}},
\end{align}
for all $t >0, x,y, h \in \R^n$, where $2|h| \leq t^{{1}/{2}} + |x-y|$.
\end{thm}
\begin{proof}
    By Lemma \ref{Davieseq} and a duality argument,
\begin{align}
\label{L2-L1bound}
    \| e^{-\phi} U_l(t,0)(\sqrt{w_t} e^{\phi} f) \|_{\L^2_w} \lesssim e^{\alpha t \rho^2} \|f\|_{\L^1_w},
\end{align}
for every $f \in \L^1_w$ and $\phi \in C^{\infty}_0(\R^n)$, where $\rho = \|\nabla_{x} \phi\|_{\L^{\infty}}$ and $\alpha$ is a positive constant depending on structural constants. {By property (ii) in Theorem \ref{fundamental}, we have $$U_l(t,0)= U_l(t,t/2)U_l(t/2,0),$$ for all $t\in \R_+.$} Hence, by combining \eqref{Davieseq} and \eqref{L2-L1bound}, we obtain
    \begin{align}
    \label{eq:Linfty-L1bound}
        \|\sqrt{w_t} e^{-\phi} U_l(t,0) (\sqrt{w_t} e^{\phi} f)\|_{\L^{\infty}} \lesssim e^{\alpha t \rho^2} \|f\|_{\L^1_w},
    \end{align}
    for every $f \in \L^1_w.$ Therefore,   by the Dunford-Pettis theorem {\cite[Thm. 1.3.2]{DP}}, there exists a kernel $K_t^{l,\phi}$ which satisfies
\begin{align*}
    \sqrt{w_t} e^{-\phi} U_l(t,0) (\sqrt{w_t} e^{\phi} f)(x) = \int_{\R^n} K_t^{l,\phi}(x,y) f(y) \, \d w(y),
\end{align*}
for all $f \in \L^1_w, \phi \in C^{\infty}_0(\R^n), x \in \R^n$. Furthermore,  $$|K^{l,\phi}_t(x,y) | \lesssim e^{\alpha t \rho^2},$$ for all $t >0, x,y \in \R^n$. Choosing $\phi =0$, a kernel $K_t^l(x,y)$ is obtained such that
\begin{align*}
      U_l(t,0) (f)(x) = \int_{\R^n} K_t^{l}(x,y) f(y) \, \d w(y),
\end{align*}
for all $f \in \L^1_w.$  Note that $K_t^{l}(x,y) = \sqrt{w_t(x)w_t(y)} e^{\phi(x)-\phi(y)}K^{l,\phi}_t(x,y)$ and hence
\begin{align}
\label{lasteq}
    |K^{l}_t(x,y)| \lesssim \frac{1}{\sqrt{w_t(x)w_t(y)}} e^{\alpha t \rho^2} e^{\phi(x)- \phi(y)},
\end{align}
for every $\phi \in C^{\infty}_0(\R^n)$ such that $\|\nabla_{x}\phi \|_{\L^{\infty}} = \rho$.  By an approximation argument we can assume that $\phi$ is a Lipschitz function in \eqref{lasteq}. Taking infimum of $\phi(x)-\phi(y)$ on \eqref{lasteq} over Lipschitz functions $\phi$ satisfying $\|\nabla_{x}\phi \|_{\L^{\infty}} = \rho$, we obtain
\begin{align*}
     |K^l_t(x,y)| \lesssim \frac{1}{\sqrt{w_t(x)w_t(y)}} e^{\alpha t \rho^2- \rho |x-y|},
\end{align*}
for all $\rho>0.$ Then, putting $\rho = \frac{|x-y|}{2\alpha t}$ concludes that
\begin{align}
\label{linfinitybound}
|K^l_t(x,y)| &\lesssim \frac{1}{\sqrt{w_t(x) w_t(y)}}
e^{-\frac{|x-y|^2}{4\alpha t}},
\end{align}
for all $x,y \in \R^n, t>0.$ Finally, \eqref{linfinitybound}, Lemma \ref{Harnack}, and an argument due to Trudinger, {see the proof of  \cite[Thm. 2.2]{Trudinger}}, imply the inequalities in \eqref{Gaussineq+}. \end{proof}

\subsection{Completing the argument: passing to the limit}

We need the following remark for the Hölder regularity of solutions.
\begin{rem}
\label{rem:Hölderregularity}
  Given $ f \in \L^2_w$, for every $l\in \R_+$ {the solution $U_{l}(t,0)f(x)$ is Hölder continuous on small closed disks $D   \subset \R^n \times \R_+$, such that $2D \subset  \R^n \times \R_+$, with bounds depending on the radius of $D$, the structural constants, and $\|U_{l}(t,0)f\|_{\L^{\infty}(2D)}$}, see \cite[Thm. B]{Ishige}. Note that in \cite[Thm. B]{Ishige} an extra assumption on $w$ is required, see property $(A5)$ in \cite[Thm. B]{Ishige}, to obtain interior Hölder regularity. However, the author uses this assumption only to derive the estimates $(3.11)$ and $(3.12)$ in \cite{Ishige}, which hold for the equation in Theorem \ref{fundamental}(iii).
\end{rem}
Now, we show that $K^l_t(x,y)$ is also Hölder continuous on compact subsets of $\R^n \times \R^n \times \R_+.$
\begin{lem}
\label{lem:Höldercon}
    For every $l \in \R_+$, the functions $K^l_t(x,y)$ is Hölder continuous on compact subsets of $\R^n \times \R^n \times \R_+$ with bounds independent of $l.$
\end{lem}
\begin{proof}
   Let fix $x \in \R^n, t, l \in \R_+$. {Define the  functions $$f_{z,r}(\cdot) := \frac{1}{w(B_r(z))} 1_{B_{r}(z)}(\cdot)$$ for every ${0<r<1}, z \in \R^n$. Then, $$U_l(t,0)f_{z,r}(x) = \int_{\R^n} K_t^l(x,y) f_{z,r}(y) \, \d w(y)$$ by Theorem \ref{Gaussian} and $U_l(t,0)f_{z,r}(x)$ is Hölder continuous on small closed disks $D   \subset \R^n \times \R_+$, such that $2D \subset  \R^n \times \R_+$, see Remark \ref{rem:Hölderregularity}, and the Hölder bounds depend on radius of $D$, the structural constants, and $\|U_{l}(t,0)f_{z,r}\|_{\L^{\infty}(2D)}$. Now, by letting $\phi\equiv 0$ in \eqref{eq:Linfty-L1bound}, we obtain
   \begin{align*}
       |U_{l}(t,0)f_{z,r}(x)|\lesssim \frac{1}{\sqrt{w(B_{\sqrt{t}}(x))}} \|f_{z,r}/\sqrt{w_t}\|_{\L^1_w} \leq \frac{1}{\sqrt{w(B_{\sqrt{t}}(x)) w(B_{\sqrt{t}/2}(z))}},
   \end{align*}
   for every $x \in \R^n, t>4 r^2.$ Consequently, $$\frac{1}{w(B_r(z))}\int_{B_r(z)} K_t^l(x,y)  \, \d w(y)$$ is Hölder continuous on compact subsets of $\R^n \times \R_+$ with bounds independent of $l,r$.} Letting $r \to 0$ and using the Lebesgue differentiation theorem, we obtain, for every fixed $z \in \R^n$, that {$K^l_t(x,z)$} is Hölder continuous on compact subsets of $\R^n \times \R_+$ with bounds independent of $l.$  Using the triangle inequality, we have
 {   \begin{align*}
      |K ^l_t(x,y)-K^l_{t+h}(x',y')| &\leq |K^l_t(x,y)-K^l_t(x,y')| + |K^l_t(x,y')-K^l_t(x',y')|  \\  & +|K^l_t(x',y')-K^l_{t+h}(x',y')|,
   \end{align*}
    for every $x,y,x',y' \in \R^n, t,h,l \in \R_+$. 
 Hence, using this we conclude the lemma by Theorem \ref{Gaussian} and the previous result that $K^l_t(x,z) $ is Hölder continuous on compact subsets of $\R^n \times \R_+$, for every fixed $z \in \R^n$, with bounds independent of $l.$}
\end{proof}
 To complete the proof of Theorem \ref{main}, we pass to the limit $l\to\infty$ in Theorem
\ref{Gaussian}. To start the argument, we first note that
\begin{align*}
    \partial_t \|U_l(t,0) f\|^2_{\L^2_w} = -2\langle \mathcal{L}_l U_l(t,0) f, U_l(t,0) f\rangle_w  \leq -2c_1 \|\nabla_x U_l(t,0) f\|_{\L^2_w}^2.
\end{align*}
Hence,
\begin{equation}
\label{eq:inequalties}
\begin{aligned}
    \|U_l(t,0) f\|^2_{\L^2_w} &\leq \|f\|^2_{\L^2_w},\\
   \int_{0}^{t} \int_{\R^n} |\nabla_x U_l(s,0) f|^2 \, \d w \d s &\leq  \frac{1}{2c_1} \| f\|^2_{\L^2_w},
\end{aligned}
\end{equation} and
\begin{equation}
\label{eq:inequalties+}
\begin{aligned}
   \int_{0}^{t} \int_{\R^n} |U_l(s,0) f|^2 \, \d w \d s &\leq  {T} \| f\|^2_{\L^2_w},
\end{aligned}
\end{equation}
for all $t\in [0,T]$.  In conclusion, up to a subsequence $ U_l(t,0) f(x)$ converges weakly to an element in $\L^2([0,T], \L^2_{w})$ as $l\to\infty$. We denote the limit $ U(t,0) f(x)$. Moreover, we have that $\{\nabla_x U_l(t,0) f\}$ converges weakly to $\nabla_x U(t,0) f$ in $\L^2([0,T], \L^2_w(\mathbb R^n,\mathbb R^n))$. As a consequence of this, \eqref{eq:inequalties}, \eqref{eq:inequalties+}, we obtain
\begin{align}\label{reg} U(t,0) f\in\L^\infty([0,T], \L^2_w({\mathbb{R}}^n))\cap \L^2((0,T], \H_w^1({\mathbb{R}}^n)),
\end{align}
and
\begin{equation}
\label{eq:L2bound}
\begin{aligned}
    \sup_{t\in [0,T] }\|U(t,0) f\|^2_{\L^2_w} +  \int_{0}^{T} \int_{\R^n} |\nabla_x U(s,0) f|^2 \, \d w \d s &\lesssim \| f\|^2_{\L^2_w},
    \\ \int_{0}^{T} \int_{\R^n} |U(s,0) f|^2 \, \d w \d s &{ \leq } T \| f\|^2_{\L^2_w}.
\end{aligned}
\end{equation}
Furthermore, $u(x,t):=U(t,0) f(x)$ is a weak solution to \begin{align}
\label{weaksol}
   \partial_t u+ \mathcal{L} u = 0\mbox{ in $\mathbb R^n\times (0,T)$}.
\end{align}

Recall that $$U_l(t,0) f(x) = \int_{\R^n } K_t^l(x,y) f(y) \, \d w(y)\mbox{ for all $(x,t)\in\mathbb R^n\times[0,T]$}.$$

Using this, the uniform boundedness and the Hölder continuity of $K^l_t(x,y)$ on compact subsets of $\R^n \times \R^n \times \R_+$ with bounds independent of $l$, see Theorem \ref{Gaussian} and Lemma \ref{lem:Höldercon}, and the Arzelà–Ascoli theorem, we conclude that there exists a $K_t(x,y)$ such that $K^l_t(x,y)$ converges, up to a subsequence, uniformly to $K_t(x,y)$ on compact subsets of $\R^n \times \R^n \times \R_+$. Also,
\begin{align}\label{cauchy+all+}
 \int_{\R^n} K_t(x,y) w(y) \d y= 1,\mbox{ for all $(x,t)\in\mathbb R^n\times (0,T)$.}
\end{align}
To prove this, note that
\begin{align*}
 &\int_{\R^n}   \frac{1}{\sqrt{w_t(x) w_t(y)}}
e^{-\frac{|x-y|^2}{4\alpha t }} \, \d w(y) \\&\leq \frac{1}{\sqrt{w_t(x)}} \biggl(\int_{B_t(x)}  \frac{1}{\sqrt{w_t(y)}}  \d w(y) + \sum_{k=1}^{\infty} e^{- \frac{2^{2(k-1)}t}{4 \alpha}}  \int_{B_{2^k t}(x) \setminus B_{2^{k-1} t}(x) } \frac{1}{\sqrt{w_t(y)}}  \d w(y) \biggr)  \\ &\lesssim  c(x,t),
\end{align*}
using \eqref{ainfw} and \eqref{eq:doubling}, where $c(x,t)$ is a constant which depends on $x$ and $t$. In conclusion, \eqref{cauchy+all+} is a result of pointwise convergence of $K^l_t(x,y)$ to $K(x,y)$ as $l \to \infty$, Theorem \ref{Gaussian}, and Lebesgue's dominated convergence theorem. Hence, by Theorem \ref{Gaussian}, there exists $c$,
$1\leq c<\infty$, and $\nu>0$, both depending only on the structural constants,  such that
\begin{equation}\label{est1+}
K_t(x,y) \leq \frac{c}{\sqrt{w_t(x) w_t(y)}}
e^{-\frac{|x-y|^2}{ct}},
\end{equation}
for all $t >0, x,y \in \R^n$, and such that
\begin{align}\label{est2+}
|K_t(x+h,y)- K_t(x,y)|&\leq \frac{c}{\sqrt{w_t(x) w_t(y)}}
\biggl(\frac{|h|}{t^{{1}/{2}} + |x-y|}\biggr)^{\nu} e^{- \frac{|x-y|^2}{ct}},\notag\\
|K_t(x,y+h)- K_t(x,y)| &\leq \frac{c}{\sqrt{w_t(x) w_t(y)}}
\biggl(\frac{|h|}{t^{{1}/{2}} + |x-y|}\biggr)^{\nu} e^{- \frac{|x-y|^2}{ct}},
\end{align}
for all $t >0, x,y, h \in \R^n$, satisfying $2|h| \leq t^{{1}/{2}} + |x-y|$.

{We next prove that
\begin{align}\label{rep}U(t,0) f(x)=\int_{\R^n } K_t(x,y) f(y) \, \d w(y)\mbox{ for all $(x,t)\in\mathbb R^n\times (0,T)$}.
\end{align}
To do this we first note, using Theorem \ref{Gaussian} and Remark \ref{rema},
\begin{align}
\label{eq:upboundrem}
    K_t^l(x,y) \lesssim \frac{1}{\sqrt{w_t(x)}} e^{-\frac{|x-y|^2}{ct}},
\end{align}
for all $x,y \in \R^n, t \in \R_+$, and
\begin{align*}
  \biggl|\int_{\R^n} \frac{e^{- \frac{|x-y|^2}{ct}}}{w_t(x)} |f(y)| \, \d w(y)\biggr|^2 \leq \frac{1}{w^2_t(x)}\biggl ( \int_{\R^n} |f(y)|^2 \, \d w(y)\biggr ) \biggl (\int_{\R^n} \frac{e^{- \frac{2|x-y|^2}{ct}}}{w(y)}  \, \d y\biggr ),
\end{align*}
for all $(x,t)\in\mathbb R^n\times\mathbb R_+$. Using \eqref{eq:doubling}, we have
\begin{align*}
\frac{1}{w^2_t(x)}\int_{\R^n} \frac{e^{- \frac{2|x-y|^2}{ct}}}{w(y)}  \, \d y&\leq \frac{1}{w^2_t(x)}\biggl(\int_{B_1(x)} \frac{1}{w(y)}\, \d y + \sum_{k = 1}^{\infty} e^{- \frac{2^{2(k-1)}}{ct}} \int_{B_{2^k}(x) \setminus B_{2^{k-1}}(x)} \frac{1}{w(y)} \, \d y   \biggr)\\
& \lesssim c(x,t),
\end{align*}
where $c(x,t)$ is a constant which depends on $x$ and $t$, and hence
\begin{align}
   \biggl|\int_{\R^n} \frac{e^{- \frac{|x-y|^2}{ct}}}{w_t(x)} |f(y)| \, \d w(y)\biggr|^2 \lesssim c(x,t)\biggl ( \int_{\R^n} |f(y)|^2 \, \d w(y)\biggr ).
\end{align}
In conclusion, by pointwise convergence of $K^l_t(x,y)$ to $K_t(x,y)$ as $l \to \infty$ and Lebesgue's dominated convergence theorem, we obtain
\begin{align}
\label{eq:pointwiseconve}
  \lim_{l \to \infty}  U_{l}(t,0)f(x) = \int_{\R^n} K_t(x,y) f(y) \, \d w(y),
\end{align}
for all $x \in \R^n.$ Let $\phi\in C_0^\infty(\mathbb R^n\times (0,T))$, let $K\subset\mathbb R^n$ be a compact set, and {let $\epsilon >0$} be such that the support of $\phi$ is contained in $K\times (\epsilon,T)$. Using \eqref{ainfw} and Lemma \ref{lem:Davieses}, we have
\begin{align*}|U_l(t,0)f(x)|  \lesssim \frac{\|f\|_{\L^2_w}}{\sqrt{w_t(x)}},
\end{align*}
for all $(x,t) \in \R^n\times \R_+$,
and
\begin{align*}
\int_0^T\int_{\R^n}  \frac{\|f\|_{\L^2_w} |\phi(x,t)|}{\sqrt{w_t(x)}} \,\d w(x)\d t\lesssim   T \tilde c(K,\epsilon)\|\phi\|_{\L^{\infty}} \|f\|_{\L^2_w},
  \end{align*}
  where $\tilde c(K,\epsilon)$ is a constant which depends on $K$ and $\epsilon$. Thus, by \eqref{eq:pointwiseconve} and Lebesgue's dominated convergence theorem, we obtain
\begin{align}
\label{eq:pointwiseconvergence}
 &\lim_{l \to \infty} \int_0^T\int_{\R^n} (U_l(t,0)f(x))\phi(x,t)\,\d w(x)  \d t\notag\\
 &= \int_{0}^T \iint_{\R^n \times \R^n } K_t(x,y) f(y) {\phi(x,t)} \, \d w(y) \d w(x) \d t
\end{align}
whenever $f\in \L^2_w$. As $ U_l(t,0) f(x)$ converges weakly to $ U(t,0) f(x)$ in  $\L^2([0,T], \L^2_{w})$, we have that
\begin{align}\label{estaconv2}\int_0^T\int_{\R^n} (U_l(t,0)f(x))\phi(x,t)\,\d w(x)\d t \to \int_0^T\int_{\R^n}  (U(t,0)f(x))\phi(x,t)\,\d w(x)\d t,
\end{align}
as $l\to\infty$ and whenever $\phi\in C_0^\infty(\mathbb R^n\times (0,T))$. Then, \eqref{estaconv2} and \eqref{eq:pointwiseconvergence} imply \eqref{rep}.}

Finally, we have to prove that $U(t,0) f(\cdot)\to f(\cdot)$ in $\L^2_w(\mathbb R^n)$ as $t\to 0^+$. Assume first that $f\in C_0^\infty(\mathbb R^n)$ with support on a ball $B \subset \R^n$. For every $c>0$, denote $cB$ as the ball keeping the center of $B$ and dilating its radius by $c$. Then, by Cauchy-Schwarz inequality, \eqref{cauchy+all+}, and Lemma \ref{off},
\begin{align*}
\|U(t,0) f-f\|^2_{\L^2_w}&\leq \int_{2B} \biggl|\int_{\R^n } K_t(x,y)|(f(y)-f(x))| \, \d w(y)\biggr|^2 \d w(x) \\& + \sum_{k=1}^{\infty} \int_{2^{k+2} B \setminus 2^{k+1} B} |U(t,0) f(x)|^2\, \d w(x) \notag\\
&\leq    \int_{2B}\biggl(\int_{\R^n } K_t(x,y)|f(y)-f(x)|^2 \, \d w(y)\biggr)  \biggl(\int_{\R^n } K_t(x,y)\, \d w(y)\biggr)     \d w(x)
\notag\\
& + \sum_{k=1}^{\infty} e^{-\frac{2^{2k}}{ct}} \int_{B} | f(x)|^2\, \d w(x)
\\&\lesssim \int_{2B} \int_{\R^n } K_t(x,y)|f(y)-f(x)|^2 \, \d w(y) \d w(x) +t \|f\|^2_{\L^2_w},
\end{align*}
for $ t \in \R_+$. As the second term on the right-hand side goes to zero as $t \to 0$, it is enough to control the first term. Now, for $t, {\delta} \in \R_+$ small enough, we have
\begin{align*}
    &\int_{2B} \int_{\R^n } K_t(x,y)|f(y)-f(x)|^2 \, \d w(y) \d w(x) \\& \leq {\delta}^2 w(2B) \|\nabla f\|^2_{\L^{\infty}}+ \int_{2B} \int_{\R^n \setminus B_{{\delta}}(x) } K_t(x,y)|f(y)-f(x)|^2 \, \d w(y) \d w(x)
\end{align*}
and, by \eqref{ainfw}, \eqref{eq:doubling}, \eqref{eq:upboundrem}, and pointwise convergence of $K^l_t(x,y)$ to $K_t(x,y)$ as $l \to \infty$, we arrive at
\begin{equation*}
\begin{aligned}
    &\int_{2B} \int_{\R^n \setminus B_{{\delta}}(x) } K_t(x,y)|f(y)-f(x)|^2 \, \d w(y) \d w(x) \\&\lesssim \|f\|^2_{\L^{\infty}}   \int_{2B} \frac{1}{w_t(x)} \int_{\R^n \setminus B_{{\delta}}(x) } e^{-\frac{|y-x|^2}{ct}}  \, \d w(y) \d w(x)
\\ & =
\|f\|^2_{\L^{\infty}}   \int_{2B} \frac{1}{w_t(x)}  \sum_{k=1}^{\infty} \int_{ B_{2^k {\delta}}(x) \setminus B_{2^{k-1} {\delta}}(x) } e^{-\frac{|y-x|^2}{ct}}  \, \d w(y) \d w(x)
    \\ &\lesssim \|f\|^2_{\L^{\infty}} \int_{2B}\frac{1}{w_t(x)} \sum_{k=1}^{\infty} e^{-\frac{2^{2(k-1)}{\delta}^2}{ct}} w(B_{2^k{\delta}}(x))\, \d w(x)
    \\ & \lesssim  \|f\|^2_{\L^{\infty}} \int_{2B}\frac{1}{w_t(x)} \sum_{k=1}^{\infty} e^{-\frac{2^{2(k-1)}}{ct}{\delta}^2} D^{k} w(B_{{\delta}}(x)) \, \d w(x) \\ & \lesssim \|f\|^2_{\L^{\infty}}|B|  \frac{ t^{m-\frac{n}{\eta}}}{{\delta}^{2m -\frac{n}{\eta} }},
\end{aligned}
\end{equation*}
where $m$ is the smallest integer, such that $m>\frac{n}{\eta}.$
 Hence, letting first $t \to 0$, we obtain
\begin{align*}
  \limsup_{t \to 0} \|U(t,0) f-f\|^2_{\L^2_w} \lesssim   {\delta}^2 w(2B) \|\nabla f\|^2_{\L^{\infty}}.
\end{align*}
Since ${\delta}$ can be arbitrarily small, we deduce
$\lim_{t \to 0} \|U(t,0) f-f\|^2_{\L^2_w} =0.$ We next use the fact that
$C_0^\infty(\mathbb R^n)$ is dense in $\L^2_w({\mathbb{R}}^n)$. Indeed, consider $f\in \L^2_w({\mathbb{R}}^n)$ and let $f_j\in C_0^\infty(\mathbb R^n)$ be such that $f_j\to f$ in $\L^2_w({\mathbb{R}}^n)$ as $j\to \infty$. We construct a solution $u_j(x,t):=U(t,0) f_j(x)$ as above for every $j$. Then, by \eqref{eq:L2bound} and the linearity
and uniqueness part of Theorem \ref{main}, we have
\begin{equation*}
\label{eq:L2bound++}
\begin{aligned}
    \sup_{t\in [0,T] }\|u-u_j\|^2_{\L^2_w} +  \int_{0}^{T} \int_{\R^n} |\nabla_x(u-u_j)|^2 \, \d w \d s &\lesssim \| f-f_j\|^2_{\L^2_w}\to 0,
    \\ \int_{0}^{T} \int_{\R^n} |u-u_j|^2 \, \d w \d s &\lesssim T \| f-f_j\|^2_{\L^2_w}\to 0,
\end{aligned}
\end{equation*}
as $j\to \infty$. Hence,
\begin{align*}
\|u(\cdot,t)-f(\cdot)\|^2_{\L^2_w}\lesssim &\|u(\cdot,t)-u_j(\cdot,t)\|^2_{\L^2_w}+\|u_j(\cdot,t)-f_j(\cdot)\|^2_{\L^2_w}+\|f_j(\cdot)-f(\cdot)\|^2_{\L^2_w}\\
\lesssim & \|u_j(\cdot,t)-f_j(\cdot)\|^2_{\L^2_w} +\|f_j(\cdot)-f(\cdot)\|^2_{\L^2_w}.
\end{align*}
Let ${\epsilon'}>0$ be small, and choose $j$ large enough so that
$$\|f_j(\cdot)-f(\cdot)\|^2_{\L^2_w}<{\epsilon'}/2.$$
With $j$ fixed, we choose ${\delta'}>0$ small enough so that
$$\|u_j(\cdot,t)-f_j(\cdot)\|^2_{\L^2_w}<{\epsilon'}/2\mbox{ for all $t\in[0,{\delta'})$}.$$
We can then conclude, for ${\epsilon'}>0$ given, that
\begin{align*}
\|u(\cdot,t)-f(\cdot)\|^2_{\L^2_w}\lesssim \epsilon\mbox{ for all $t\in[0,{\delta'})$}.
\end{align*}
This proves that $U(t,0) f\to f$ in $\L^2_w(\mathbb R^n)$ as $t\to 0^+$, whenever $f\in\L^2_w({\mathbb{R}}^n)$. The proof of Theorem
\ref{Gaussian} is therefore complete.

\section{Declaration}

\noindent Ethical Approval: Not applicable.

\noindent Competing interests: No competing interests

\noindent Authors' contributions: Both authors have contributed in the writing and reviewing of the article.

\noindent Funding: K.N. was partially supported by grant 2022-03106
from the Swedish research council (VR).

\noindent Availability of data and materials: Not applicable.

\def\cprime{$'$} \def\cprime{$'$} \def\cprime{$'$}


\begin{thebibliography}{10}
\providecommand{\url}[1]{{\tt #1}}
\providecommand{\urlprefix}{URL}
\providecommand{\eprint}[2][]{\url{#2}}




\bibitem{Ar67}
D.~G. Aronson. \emph{Bounds for the fundamental solution of a parabolic
  equation}. Bull. Amer. Math. Soc. \textbf{73} (1967), 890--896.

\bibitem{Ar68}
D.~G. Aronson. \emph{Non-negative solutions of linear parabolic equations}. Ann.
  Scuola Norm. Sup. Pisa (3) \textbf{22} (1968), 607--694.


\bibitem{AAEN}
\textsc{A. Ataei}, \textsc{M. ~Egert}, and \textsc{K.~Nystr{\"o}m}.
\newblock {\em The Kato square root problem for weighted parabolic operators\/}. To appear in Analysis $\&$ PDE.


\bibitem{AEN}
\textsc{P.~Auscher}, \textsc{M.~Egert}, and \textsc{K.~Nystr\"om}.
\newblock {\em The Dirichlet problem for second order parabolic operators in divergence form}.
\newblock J. \'Ecole Polytech. Math. \textbf{5} (2018), 407--441.

\bibitem{CS}
\textsc{E.~Chiarenza} and \textsc{R.~Serapioni}. \newblock{\em A remark on a Harnack inequality for degenerate parabolic
equations\/}. Rend. Sem. Mat. Univ. Padova \textbf{73} (1985), 179--190.

\bibitem{CUR0}
 \textsc{D.~Cruz-Uribe} and \textsc{C.~Rios}. \newblock{\em Gaussian bounds for degenerate parabolic equations.\/}
 \newblock{J. Funct. Anal. \textbf{255} (2) (2008) 283–-312.}


\bibitem{CUR1}
\textsc{D.~Cruz-Uribe} and \textsc{C.~ Rios}.
\newblock{\em Corrigendum to "Gaussian bounds for degenerate parabolic equations''\/}.
J. Funct. Anal. \textbf{255} (2) (2008) 283–312.


\bibitem{CUR2}
\textsc{D.~Cruz-Uribe} and \textsc{C.~ Rios}.
\newblock{\em The Kato problem for operators with weighted ellipticity.}
\newblock{Trans. Amer. Math. Soc. \textbf{367} (2015), no. 7, 4727–4756.}

\bibitem{Davies}
\textsc{E.~Davies.}
\newblock{Heat kernel bounds, conservation of probability and the Feller property.}
\newblock{J. Anal. Math. \textbf{58} (1992), 99–119.}

\bibitem{DP}
\textsc{N.~Dunford} and \textsc{B.~Pettis}.
\newblock{\em Linear operations on summable functions.}
\newblock{Trans. Amer. Math. Soc. \textbf{47} (1940), 323–392.}

 \bibitem{Evans}
\textsc{C. Evans}.
\newblock {\em Partial differential equations\/}. Textbook, 2nd edition.

\bibitem{FaSt}
\textsc{E.~B. Fabes} and  \textsc{D.~W. Stroock}. \emph{A new proof of {M}oser's parabolic
 {H}arnack inequality using the old ideas of {N}ash}. Arch. Rational Mech.
  Anal. \textbf{96} (1986), no.~4, 327--338.


\bibitem{F}
\textsc{A.~Friedman}.
\newblock{ Partial differential equations of parabolic type.}
\newblock{Prentice-Hall, Inc., Englewood Cliffs, N.J. (1964)}


\bibitem{Ishige}
\textsc{K.~Ishige}.
\newblock{\em On the behavior of the solutions of degenerate parabolic equations.}
\newblock{Nagoya Math. J. \textbf{155} (1999), 1–26.}


\bibitem{IKO}
\textsc{K.~Ishige},\textsc{Y.~Kabeya}, and \textsc{E.M.~Ouhabaz}. \newblock{\em The heat kernel of a Schrödinger operator with inverse square potential.} Proc. London Math. Soc. \textbf{115} (2017), 381–410.


\bibitem{Kato}
\textsc{T.~Kato}.
\newblock{\em Abstract evolution equations of parabolic type in Banach and Hilbert spaces.}
\newblock{Nagoya Math. J. \textbf{19} (1961), 93–125.}

\bibitem{Kilp}
\textsc{T. Kilpel{\"a}inen}. \newblock{\em Weighted Sobolev spaces and capacity.}
\newblock Ann. Acad. Sci. Fenn. Ser. A I. Math. \textbf{19} (1994), 95--113.

\bibitem{LN}
\textsc{M.~Litsgård} and \textsc{K.~ Nyström}. \newblock{\em On local regularity estimates for fractional powers of parabolic operators with time-dependent measurable coefficients.}\newblock{ Journal of Evolution
Equations. \textbf{23} (3) (2023) https://doi.org/10.1007/s00028-022-00844-0}

\bibitem{Na}
\textsc{J.~Nash}. \emph{Continuity of solutions of parabolic and elliptic equations}.
 Amer. J. Math. \textbf{80} (1958), 931--954.


\bibitem{Stein}
\textsc{E.M. Stein}.
\newblock{ Harmonic Analysis: Real-Variable Methods, Orthogonality, and Oscillatory Integrals:\/}
\newblock Monographs in Harmonic Analysis, III, Princeton Mathematical
Series, vol. 43, Princeton University Press, Princeton, NJ, (1993).



\bibitem{Trudinger}
\textsc{N.~Trudinger}.
\newblock{\em Pointwise estimates and quasilinear parabolic equations.}
\newblock{Comm. Pure Appl. Math. \textbf{21} (1968), 205–226.}




\end{thebibliography}
\end{document}